\documentclass{amsa}

\usepackage{comment}
\usepackage[dvipdfmx]{graphicx}
\usepackage{here}
\usepackage{graphicx}
\usepackage{amsmath, amsfonts, amssymb}
\usepackage{xcolor}
\usepackage{bigints}
\usepackage{enumerate}

\usepackage{indentfirst}
\usepackage{mathrsfs}
\usepackage{mathtools}
\usepackage{empheq}
\usepackage{bm}

\numberwithin{equation}{section}
\numberwithin{figure}{section}
\numberwithin{table}{section}

\usepackage{amsthm}
\theoremstyle{plain}
\newtheorem{theorem}{Theorem}[section]
\newtheorem{lemma}{Lemma}[section]
\newtheorem{proposition}{Proposition}[section]

\theoremstyle{definition}
\newtheorem{definition}{Definition}[section]
\newtheorem{remark}{Remark}[section]

\usepackage{tikz}

\graphicspath{{images/}}

\allowdisplaybreaks 

\begin{document}
\label{page:t}
\thispagestyle{plain}
\title{
{CONVERGENCE OF APPROXIMATE SOLUTIONS CONSTRUCTED BY THE  FINITE VOLUME METHOD FOR THE MOISTURE TRANSPORT MODEL IN POROUS MEDIA}
}
\author{Akiko Morimura\footnotemark[1]}
\affiliation{Division of Mathematical and Physical Science, Graduate School of Science,\\Japan Women's University\\
2-8-1 Mejirodai, Bunkyoku, Tokyo, 112-8681, Japan}
\email{m1716096ma@ug.jwu.ac.jp}
\sauthor{Toyohiko Aiki} 
\saffiliation{Department of Mathematics, Physics and Computer Science, Faculty of Science,\\Japan Women's University\\
2-8-1 Mejirodai, Bunkyoku, Tokyo, 112-8681, Japan}
\semail{aikit@fc.jwu.ac.jp}
\footcomment{
This work is partially supported by Ebara Corporation.\\
AMS Subject Classification: 76S05, 35K55, 65M08, 35K59.\\
Keywords: moisture transport, porous media, pseudo monotone, finite volume method.
}
\footnotetext[1]{Corresponding author}
\maketitle
\vspace{-0.5cm}
\noindent
{\bf Abstract.}
We consider the initial-boundary value problem for a nonlinear parabolic equation in the one-dimensional interval.
This problem is motivated by a mathematical model for moisture transport in porous media.
We establish the uniqueness of weak solutions to the problem by using the dual equation method.
Moreover, we prove the convergence of approximate solutions constructed with the finite volume method.
\newpage
\section{Introduction}
This paper is concerned with the following initial-boundary value problem for a nonlinear parabolic equation of the form
\begin{align}
\label{P_eq}& \partial_t h(v) = \partial_x \left( \partial_x v + b(v) p \right)
&&\text{in $Q(T)$}, \\
\label{P_BC}& \partial_x v + b(v) p  = 0
&&\text{at $x = 0, 1$ and for $0 < t < T$},\\
\label{P_IC}& v(0, x) = v_0(x)
&&\text{for $x \in (0, 1)$}, 
\end{align}
where $Q(T) \coloneqq (0, T) \times (0, 1)$ with $T > 0$, $v \colon [0, T) \times (0, 1) \to \mathbb{R}$ is an unkown function, $v_{0} \colon (0, 1) \to \mathbb{R}$ is an initial datum, and $p \colon (0, T) \times [0, 1] \to \mathbb{R}$ is a given function.
In addition, given functions $h, b \colon \mathbb{R} \to \mathbb{R}$ satisfy the following assumptions:
\begin{itemize}
\item[(A1)]
$h \in C^2(\mathbb{R})$ with $\delta_{h} \leq h' \leq C_{h}$ and $|h''| \leq C_{h}$ on $\mathbb{R}$; 
\item[(A2)]
$b \in C^2(\mathbb{R})$ with $\delta_{b} \leq b \leq C_{b}$ and $|b'|,  |b''| \leq  C_{b}$ on $\mathbb{R}$,
\end{itemize}
where $\delta_{h}$,  $\delta_{b}$,  $C_{h}$, and $C_{b}$ are positive constants.
Throughout this paper, (P)($p, v_0$) denotes the initial-boundary value problem \eqref{P_eq}--\eqref{P_IC}.

This problem is strongly motivated by the mathematical model for moisture transport in porous media proposed by Green--Dabiri--Weinaug--Prill \cite{G-D-W-P} (GDWP model). 
In Fukui--Iba--Hokoi--Ogura \cite{F-I-H-O}, the validity of the GDWP model was confirmed by comparing experimental results with numerical simulations.
To provide a basis for determining the analytical validity of the GDWP model, in our previous work \cite{M-A}, we introduced the following initial-boundary value problem:
\begin{align}
\label{PP_eq}
&\partial_t \psi(u) = \partial_x \left( \lambda(u) \partial_x ( u +  P)  \right) 
&&\text{in $Q(T)$}, \\
\label{PP_BC}
&\partial_x   ( u + P) = 0 
&&\text{at $x = 0, 1$ and for $0 < t < T$},\\
\label{PP_IC}
&u(0, x) = u_0(x)
&&\text{for $x \in (0, 1)$}, 
\end{align}
which is a simplified model of the GDWP model. 
Here, $u \colon [0, T) \times (0, 1) \to \mathbb{R}$ is an unknown function, $u_{0} \colon (0, 1) \to \mathbb{R}$ is an initial datum, and $\psi, \lambda \colon \mathbb{R} \to \mathbb{R}$ are given functions.
We note that $u$, $\psi(u)$, and $\lambda(u)$ describe the water chemical potential, the water mass, and the diffusion coefficient, respectively.
In addition, we have set $P (t, x) \coloneqq P_a (t, x)/\rho_w$, where $P_a \colon (0, T) \times [0, 1] \to \mathbb{R}$ is the air pressure and $\rho_w$ is the water density.
Formally, by setting $v \coloneqq \widehat{\lambda} \circ u$, $v_0 \coloneqq \widehat{\lambda} \circ u_0$, $p \coloneqq \partial_x P$, $h \coloneqq \psi \circ \widehat{\lambda}^{-1}$, and $b \coloneqq \lambda \circ \widehat{\lambda}^{-1}$, the initial-boundary value problem \eqref{PP_eq}--\eqref{PP_IC} turns into (P)($p, v_0$).
Here, $\widehat{\lambda}$ denotes the primitive of $\lambda$. 

In \cite{F-I-H-O, G-D-W-P}, the function $\psi$ was assumed to be continuous and bounded.
Under this assumption, the equation \eqref{PP_eq} is classified as being of elliptic-parabolic type.
The standard form of the initial-boundary value problems for elliptic-parabolic equations is given by
\begin{align}
\label{EPP_eq}
&\partial_t \beta(w) = \mathrm{div} \; \alpha (t, x, w, \nabla w)
&&\text{in $(0, T) \times \Omega$},\\
\label{EPP_BC}
&\alpha(t, x, w, \nabla w) = 0
&&\text{on $(0, T) \times \partial \Omega$},\\
\label{EPP_IC}
&w(0) = w_0 
&&\text{on $\Omega$},
\end{align}
where $d \in \mathbb{N}$; $\Omega \subset \mathbb{R}^d$ is a domain with the sufficiently regular boundary $\partial \Omega$; $w \colon [0, T) \times \Omega \to \mathbb{R}$ is an unknown funciton; $\beta \colon \mathbb{R} \to \mathbb{R}$ and $\alpha \colon (0, T) \times \overline{\Omega} \times \mathbb{R} \times \mathbb{R}^d \to \mathbb{R}$ are given functions. 
For this problem \eqref{EPP_eq}--\eqref{EPP_IC}, Alt--Luckhaus \cite{A-L} established the existence and uniqueness of weak solutions, provided that $\beta$ is monotone and that $\alpha$ is of the form $\alpha = \alpha(\beta(w), \nabla w)$ and in particular independent of $(t, x) \in (0, T) \times \overline{\Omega}$.
Moreover, Kenmochi--Pawlow \cite{K-P-1, K-P-2} investigated an obstacle problem for \eqref{EPP_eq}--\eqref{EPP_IC}, under the assumption $\alpha = \alpha(\nabla w)$.
On the other hand, in our problem (P)($p, v_0$) or the problem \eqref{PP_eq}--\eqref{PP_IC}, it is necessary to consider the case of $\alpha = \alpha(t, x, w, \nabla w)$ or $\alpha = \alpha(t, x, \beta(w), \nabla w)$, due to the appearance of the function $p$ depending on $(t, x) \in (0, T) \times [0, 1]$.
However, as mentioned below, the presence of $p$ makes significant difficulties in the analysis of our model. 
To overcome this problem, we impose the strong assumptions (A1) and (A2) on $h$ and $b$, respectively.

To clarify the novelty of this paper, we review several related results.
Equations of the form \eqref{PP_eq} are known as Richards-type equations and have been extensively studied (cf. Stokke--Mitra--Storvik--Both--Radu \cite{S-M-S-B-R} and Mitra--Vohral\'ik \cite{M-V}).
In most of these works, the function $P$ in \eqref{PP_eq} is assumed to be independent of the time variable.
In contrast, our target equation \eqref{P_eq}, which is connected with the equation \eqref{PP_eq} via a certain transformation, has the time-dependent function $p$.
Hence, the known results mentioned above cannot be applied directly.

The problem (P)($p, v_0$) can be rewritten as the form \eqref{EPP_eq}--\eqref{EPP_IC} by setting $\beta(r) \coloneqq r$ and $\alpha(t, x, r, s) \coloneqq s/h'(h^{-1}(r)) + b(h^{-1}(r))p(t, x)$.
Here, we have used the fact that $h$ has the inverse by virtue of the assumption (A1).
For this type of equation, Roub\'{i}\v{c}ek \cite{Roubicek} established the existence of weak solutions from the perspective of pseudo-monotone operators.
Therefore, the existence of weak solutions to (P)($p, v_0$) follows from \cite[Proposition 8.37]{Roubicek}.
As for the uniqueness of weak solutions to \eqref{EPP_eq}--\eqref{EPP_IC}, we note that Otto \cite{Otto} considered the case where $\alpha$ is independent of $(t, x) \in (0, T) \times \overline{\Omega}$ and that Feo \cite{Feo} studied the case where $\alpha$ is locally Lipschitz continuous with respect to $r$ under the Dirichlet boundary condition.
Chiyo--Terasaki--Tsuzuki--Yokota \cite{C-T-T-Y} proved the uniqueness of weak solutions to the problem \eqref{PP_eq}--\eqref{PP_IC} in the case where $\lambda/\psi'$ is a positive constant.
In this paper, we establish the uniqueness of weak solutions to (P)($p, v_0$) even in the case corresponding to $\lambda/\psi'$ not being constant (see Theorem \ref{main_thm} in the next section).

Although this paper focuses on weak solutions of (P)($p, v_0$), our goal is to construct strong solutions. 
From the viewpoint of physics, the existence of strong solutions to (P)($p, v_0$) is crucial for ensuring that the equation \eqref{P_eq} and the boundary condition \eqref{P_BC} hold pointwise.
To the best of our knowledge, there is no result on the existence of the strong solutions.
For nonlinear problems, the finite element method (FEM) has been widely used.
However, for the problem (P)($p, v_0$), it seems impossible to construct strong solutions by using the FEM, due to the lack of the regularity of approximate solutions.
On the other hand, there is room to apply the finite volume method to overcome this problem.

The finite volume method (FVM) is one of the well-known numerical discretization techniques for partial differential equations, as well as the finite element method.
One advantage of the FVM over the FEM is that test functions can be taken more flexibly to derive discrete energy estimates.
Thanks to this advantage, approximate solutions may have a sufficient regularity to ensure the existence of strong solutions to (P)($p, v_0$).
As a basis for this attempt, we show that approximate solutions constructed by using the FVM converge to the weak solution to (P)($p, v_0$) (see Theorem \ref{AS_conv} in the next section).
For related results in this direction, we refer to Canc\`{e}s--Guichard \cite{C-G} and Oulhaj--Canc\`{e}s--Chainais-Hillairet \cite{O-C-C}.
The authors considered the convergence of approximate solutions, constructed with the FVM, for the problem \eqref{PP_eq}--\eqref{PP_IC}.
However, since they assumed that $P$ in \eqref{PP_eq} is independent of the time variable, their results do not cover our problem.

We emphasize that concerning the problem (P)($p, v_0$), the FVM is effective to prove not only the convergence of approximate solutions but also the existence of weak solutions.
As far as we know, the known approach to establishing the existence is based on the pseudo-monotone theory or the Galerkin method.
In the proof of our convergence result, the existence of weak solutions to (P)($p, v_0$) is proven in the framework of the FVM.

The rest of this paper is organized as follows.
In the next section, we state our main results: Theorems \ref{main_thm} and \ref{AS_conv}.
In Sections \ref{P_proof_u} and \ref{AS_proof}, we give the proofs of Theorems \ref{main_thm} and \ref{AS_conv}, respectively.

\section{Notation and results} \label{NAR}
In this paper, we put
\begin{align*}
H \coloneqq L^{2}(0, 1), \qquad X \coloneqq H^{1}(0, 1)
\end{align*}
with the standard norms denoted by $|\cdot|_{H} \coloneqq |\cdot|_{L^{2}(0, 1)}$ and $|\cdot|_{X} \coloneqq |\cdot|_{H^{1}(0, 1)}$, respectively.

We first define a weak solution of (P)($p, v_0$).

\begin{definition} \label{P_sol_def}
Let $T>0$.
A function $v \colon [0, T) \times (0, 1) \to \mathbb{R}$ is said to be a weak solution of (P)($p, v_0$) if it satisfies
\begin{align}
\label{P_D1}
v \in  L^{\infty}(0, T; H) \cap L^2(0, T; X)
\end{align}
with the identity
\begin{align}
\label{P_D2}
&-\int_{Q(T)}  h(v (t, x)) \partial_t \eta (t, x) dxdt + \int_{Q(T)} (\partial_x v (t, x) + b(v (t, x)) p (t, x)) \partial_x \eta (t, x) dxdt \nonumber \\
&\hspace{20mm} = \int_0^1 h(v_0(x)) \eta(0, x) dx 
\end{align}
for any $\eta \in W^{1, 2}(0, T; H) \cap L^2(0, T; X)$ with $\eta(T, \cdot) = 0$.
\end{definition}

The existence of weak solutions to (P)($p, v_0$) was established by \cite{Roubicek}.

\begin{proposition}[{\cite[Proposition 8.37]{Roubicek}}] \label{P_e}
Let $T>0$.
Assume (A1) and (A2).
If $p \in L^2(0,T; H)$ and $v_0 \in H$, then {\rm(P)($p, v_0$)} has a weak solution.
\end{proposition}

The main result on the uniqueness is stated as follows.

\begin{theorem}\label{main_thm}
Let $T > 0$.
Assume (A1) and (A2).
If $p \in L^{4}(0, T; H)$ and $v_0 \in H$, then {\rm(P)($p, v_0$)} has at most one weak solution.
\end{theorem}

\begin{remark}
For the uniqueness of weak solutions to (P)($p, v_0$), Theorem \ref{main_thm} improves the result by Chiyo--Terasaki--Tsuzuki--Yokota \cite{C-T-T-Y}.
In fact, the authors in \cite{C-T-T-Y} assumed $p \in L^\infty(0,T; L^\infty(0,1)) \cap L^2(0,T; H)$.
On the other hand, we assume only $p \in L^4(0, T; H)$.
\end{remark}

We next construct approximate solutions converging to the weak solution $v$ of (P)($p, v_0$) given in Proposition \ref{P_e}.
Firstly, we divide the interval $(0, 1)$ into $n$ equal parts and set $\Delta x^{(n)} = 1/n$.
Based on the FVM, we discretize (P)($p, v_0$) in space and derive the following Cauchy problem for a system of ordinary differential equations with respect to the time variable:
\begin{align}
\label{OP_eq}
h'(v_i^{(n)})\frac{d}{dt}v_{i}^{(n)} &= 
\begin{dcases}
\frac{1}{\Delta x^{(n)}} \left( \frac{v_2^{(n)} - v_1^{(n)}}{\Delta x^{(n)}} + b\left( \frac{v_1^{(n)} + v_2^{(n)}}{2} \right) p_{1}^{(n)} \right),  &\qquad i = 1, \\
\frac{1}{\Delta x^{(n)}} \left( \frac{v_{i+1}^{(n)} - v_i^{(n)}}{\Delta x^{(n)}} + b\left( \frac{v_i^{(n)} + v_{i+1}^{(n)}}{2} \right) p_{i}^{(n)} \right. \\
\qquad \left. - \left( \frac{v_i^{(n)} - v_{i-1}^{(n)}}{\Delta x^{(n)}} + b\left( \frac{v_{i-1}^{(n)} + v_{i}^{(n)}}{2} \right) p_{i-1}^{(n)} \right)\right),  &\qquad i = 2, \ldots, n-1, \\
-\frac{1}{\Delta x^{(n)}} \left( \frac{v_n^{(n)} - v_{n-1}^{(n)}}{\Delta x^{(n)}} + b\left( \frac{v_{n-1}^{(n)} + v_n^{(n)}}{2} \right) p_{n-1}^{(n)} \right),  &\qquad i = n, 
\end{dcases}\\
\label{OP_IC}
v_i^{(n)}(0) &= v_{0, i}^{(n)}, \qquad i = 1, 2, \ldots, n,
\end{align}
where
\begin{align*}
v_{i}^{(n)}(t) &\coloneqq \frac{1}{\Delta x^{(n)}} \int_{(i-1)\Delta x^{(n)}}^{i \Delta x^{(n)}} v(t, \xi)  d\xi,\\
p_{i}^{(n)}(t) &\coloneqq \frac{1}{\Delta x^{(n)}} \int_{(i-1)\Delta x^{(n)}}^{i \Delta x^{(n)}} p (t, \xi) d\xi,\\
v_{0, i}^{(n)} &\coloneqq \frac{1}{\Delta x^{(n)}} \int_{(i-1)\Delta x^{(n)}}^{i \Delta x^{(n)}} v_0(\xi)  d\xi
\end{align*}
for $i = 1, 2, \ldots, n$.
Let (OP)$^{(n)} (\bm{p}^{(n)}, \bm{v}_0^{(n)})$ denote the Cauchy problem \eqref{OP_eq} and \eqref{OP_IC}, where $\bm{p}^{(n)} \coloneqq (p_{1}^{(n)}, \ldots, p_{n}^{(n)})$ and $\bm{v}_0^{(n)} \coloneqq (v_{0, 1}^{(n)}, \ldots, v_{0, n}^{(n)})$.
Secondly, we approximate $p$ by $\rho_\delta \coloneqq J_\delta^{(2)} \ast p$ for $\delta > 0$, where $J_\delta^{(2)}$ is a mollifier on $\mathbb{R}^2$, $\ast$ is the convolution in $\mathbb{R}^2$, and $p$ is defined on $\mathbb{R}^2$ by zero extention.
We put
\begin{align*}
\rho_{\delta, i}^{(n)}(t) \coloneqq \frac{1}{\Delta x^{(n)}} \int_{(i-1)\Delta x^{(n)}}^{i \Delta x^{(n)}} \rho_\delta(t, \xi)  d\xi \qquad \text{for $i = 1, 2, \ldots, n$}
\end{align*}
and set $\bm{\rho}_\delta^{(n)} \coloneqq (\rho_{\delta, 1}^{(n)}, \ldots, \rho_{\delta, n}^{(n)})$.
By a simple computation, for each $n \in \mathbb{N}$ and $\delta > 0$, (OP)$^{(n)}(\bm{\rho}_{\delta}^{(n)}, \bm{v}_{0}^{(n)})$ has a unique solution $(v_{\delta, 1}^{(n)}, \ldots, v_{\delta, n}^{(n)}) \in (C^1(\mathbb{R}))^n$. 
By using this solution, we set
\begin{align}
\label{NS} v_{\delta}^{(n)} (t, x) \coloneqq \sum^{n}_{i=1}  \chi_i^{(n)}(x) v_{\delta, i}^{(n)}(t) 
\qquad \text{for $(t,x) \in Q(T)$},
\end{align}
where $\chi_i^{(n)}$ is the characteristic function of the interval $[ (i-1)\Delta x^{(n)}, i \Delta x^{(n)})$. Finally, as an approximation of $\partial_x v$, we introduce a divided difference $\widetilde{v}_{\delta}^{(n)}$ by
\begin{align*}
\widetilde{v}_{\delta}^{(n)} (t, x) \coloneqq
\begin{dcases}
0, &\qquad 0 \leq x < \Delta x^{(n)},\\
\frac{v_{\delta}^{(n)}(t, x) - v_{\delta}^{(n)}(t, x - \Delta x^{(n)})}{\Delta x^{(n)}}, &\qquad \Delta x^{(n)} \leq x \leq 1. 
\end{dcases}
\end{align*}
Under the above preparation, we state the main result on the convergence of approximate solutions.

\begin{theorem} \label{AS_conv}
Under the same assumption as in Proposition \ref{P_e}, for each $\delta > 0$, there exists a weak solution $v_{\delta}$ to {\rm(P)($\rho_\delta, v_0$)} such that
\begin{alignat*}{2}
v_{\delta}^{(n)} &\to v_\delta
&&\qquad \text{strongly in $L^2(0, T; H)$ and
 weakly$\ast$ in $L^\infty(0, T; H)$}, \\
\widetilde{v}_{\delta}^{(n)} &\to \partial_x v_\delta 
&&\qquad \text{weakly in $L^2(0, T; H)$}
\end{alignat*}
as $n \to \infty$.
Moreover, there exist a subsequence $\{ \delta_j \}$ of $\{ \delta \}$ and a weak solution $v$ to {\rm(P)($p, v_0$)} such that 
\begin{alignat*}{2}
v_{\delta_j} &\to v 
&&\quad \text{strongly in $L^2(0, T; H)$,  
 weakly in $L^2(0, T; X)$, and  
 weakly$\ast$ in $L^\infty(0, T; H)$}
\end{alignat*}
as $j \to \infty$.
If $p \in L^4(0, T; H)$ additionally, then the above convergence holds for the whole sequence.
\end{theorem}

In the proof of Theorem \ref{AS_conv}, the discretized Aubin-Lions compactness theorem, established by Gallou\"{e}t--Latch\'{e} \cite{G-L}, plays a crucial role (see Lemma \ref{discrete_Aubin} in Section \ref{AS_proof}). 
We remark that Theorem \ref{AS_conv} provides an alternative proof based on the FVM.

At the end of this section, we give the following preliminary lemma, which follows from straightforward computations.

\begin{lemma}\label{lem:1DGNineq}
For any $u \in X$, it holds that
\begin{align}
\label{1DGNiineq_Linf} |u|^2_{L^\infty(0,1)}
\leq |u|_H^2 + 2|u|_H |\partial_x u|_H.
\end{align}
\end{lemma}

\section{Uniqueness of weak solutions} \label{P_proof_u}
In this section, we give the proof of Theorem \ref{main_thm} by applying the dual equation method developed in Ladyzenskaja--Solonnikov--Ural'ceva \cite{L-S-U} and Niezgodka--Pawlow \cite{N-P}. 
Throughout this section, we suppose the same assumption as in Theorem \ref{main_thm}.

Let $v_1$ and $v_2$ be weak solutions of (P)($p, v_0$).
We put $z_1 \coloneqq h(v_1)$, $z_2 \coloneqq h(v_2)$, and $z \coloneqq z_1 - z_2$.
In addition, we define functions $\sigma: Q(T) \to \mathbb{R}$ and $q: Q(T) \to \mathbb{R}$ as follows:
\begin{align*}
\sigma (t,x)
&\coloneqq
\begin{dcases}
1, &\qquad z (t,x) = 0, \\
\frac{h^{-1}(z_1 (t,x)) - h^{-1}(z_2 (t,x))}{z (t,x)}, &\qquad \text{otherwise},\\
\end{dcases}\\
q (t,x)
&\coloneqq
\begin{dcases}
1, &\qquad z (t,x) = 0, \\
\frac{b(h^{-1}(z_1 (t,x))) - b(h^{-1}(z_2 (t,x)))}{z (t,x)}, &\qquad \text{otherwise}.
\end{dcases}
\end{align*}
We note that $h$ has the inverse by virtue of (A1).
It follows from (A1) and (A2) that there exist positive constants $\delta_\sigma, C_\sigma$, and $C_q$ such that 
\begin{align} \label{sigma_q_esti}
\delta_\sigma \leq \sigma \leq C_\sigma, 
\qquad 0 \leq q \leq C_q 
\qquad \text{a.e. on $Q(T)$}.
\end{align}
Let $\zeta \in L^2(0,T; H^2(0,1)) \cap W^{1,2}(0,T; H)$ satisfy $\partial_x \zeta (\cdot, 0) = \partial_x \zeta (\cdot, 1) = 0$ in $(0, T)$ and $\zeta(T, \cdot) = 0$ in $(0, 1)$.
By substituting $\eta = \zeta$ to \eqref{P_D2} and using integration by parts, we have
\begin{align}
\label{B2_diff}
- \int^T_0 \int^1_0 z (\partial_t \zeta - \sigma \partial_x^2 \zeta - \varphi \partial_x \zeta) dx dt = 0,
\end{align}
where $\varphi \coloneqq q p$.
We remark that $\varphi \in L^4(0,T; H)$.
For $\varepsilon > 0$, we set $\sigma_\varepsilon \coloneqq J_\varepsilon^{(2)} \ast \sigma$.
In addition, we define
\begin{align*}
\varphi_\varepsilon (t, x) \coloneqq \int_\mathbb{R} J_\varepsilon^{(1)}(t-s) \int_\mathbb{R} J_\varepsilon^{(1)}(x-y) \varphi(s, y) dy ds \qquad \text{for $(t, x) \in \mathbb{R}^2$},
\end{align*}
where $J_\varepsilon^{(1)}$ is a mollifier on $\mathbb{R}$ and the function $\varphi$ is extended to $\mathbb{R}^2$ by zero.
From the basic properties of the mollifier, we see that $\sigma_\varepsilon$ and $\varphi_\varepsilon$ are continuous on $\overline{Q(T)}$ and satisfy
\begin{align}
\label{sigma_esti}
&\delta_\sigma \leq \sigma_\varepsilon \leq C_\sigma \qquad \text{on $Q(T)$},\\ 
\label{varphi_esti}
&|\varphi_\varepsilon|_{L^2(Q(T))} \leq C_q |p|_{L^2(Q(T))}, \qquad
|\varphi_\varepsilon|_{L^4(0,T; H)} \leq C_q |p|_{L^4(0,T; H)}
\qquad \text{for $\varepsilon > 0$},
\end{align}
where $\delta_\sigma$ and $C_\sigma$ are the positive constants in \eqref{sigma_q_esti}.
Moreover, it holds that $\sigma_\varepsilon \to \sigma$ in $L^4(Q(T))$ and $\varphi_\varepsilon \to \varphi$ in $L^4(0,T; H)$ as $\varepsilon \to 0$.
For every $\varepsilon > 0$ and $\xi \in C^\infty_0(Q(T))$, let (DP)($\varepsilon, \xi$) denote the following initial-boundary value problem:
\begin{align}
\label{DP_eq}
&\partial_t \zeta_{\varepsilon} - \sigma_\varepsilon \partial_x^2 \zeta_{\varepsilon} - \varphi_\varepsilon \partial_x \zeta_{\varepsilon} = \xi 
&&\text{in $Q(T)$}, \\
\label{DP_BC}
&\partial_x \zeta_{\varepsilon}(t, x) = 0 
&&\text{at $x = 0,1$ and for $0 < t < T$},\\
\label{DP_IC} &\zeta_{\varepsilon}(0, x) = 0 
&&\text{for $x \in (0,1)$}.
\end{align}
The problem (DP)($\varepsilon, \xi$) is an approximation of the time reversal problem obtained from \eqref{B2_diff}.
In order to solve (DP)($\varepsilon, \xi$), for fixed $\widetilde{\zeta} \in L^2(0,T; X)$, we introduce an auxiliary problem \eqref{DP_BC}--\eqref{DP_eq_aux}, denoted by (AP)($\varepsilon, \xi, \widetilde{\zeta}$):
\begin{align}
\label{DP_eq_aux}
\partial_t \zeta_{\varepsilon} - \sigma_\varepsilon \partial_x^2 \zeta_{\varepsilon} - \varphi_\varepsilon \partial_x \widetilde{\zeta} = \xi
\qquad \text{in $Q(T)$}.
\end{align}

The existence and uniqueness of strong solutions to (AP)($\varepsilon, \xi, \widetilde{\zeta}$) follows from the standard argument, since \eqref{DP_eq_aux} is linear and $\sigma_\varepsilon$ is smooth and strictly positive.

\begin{lemma} \label{AP_lem}
Let $\varepsilon > 0$ and let $\xi \in C^\infty_0 (Q(T))$.
If $\widetilde{\zeta} \in L^2(0,T; X)$, then {\rm (AP)($\varepsilon, \xi, \widetilde{\zeta}$)} has a unique strong solution $\xi_\varepsilon \in W^{1,2} (0,T; H) \cap L^\infty (0,T; X) \cap L^2(0,T; H^2(0,1))$.
\end{lemma}

For each $\varepsilon > 0$, we define a mapping $\Gamma_\varepsilon \colon L^2(0, T; X) \to L^2(0, T; X)$ by $\Gamma_\varepsilon (\widetilde{\zeta}) \coloneqq \zeta_{\varepsilon}$ for $\widetilde{\zeta} \in L^2(0,T; X)$, where $\zeta_{\varepsilon}$ is the strong solution of (AP)($\varepsilon, \xi, \widetilde{\zeta}$) given by Lemma \ref{AP_lem}.

\begin{lemma} \label{DP_lem_1}
Let $\varepsilon > 0$, $\xi \in C^\infty_0 (Q(T))$, and $\widetilde{\zeta}_i \in L^2(0,T; X)$ with $i = 1, 2$.
Define $\widetilde{\zeta} \coloneqq \widetilde{\zeta}_1 - \widetilde{\zeta}_2$ and $\zeta_{\varepsilon} \coloneqq \zeta_{\varepsilon, 1} - \zeta_{\varepsilon, 2}$, where $\zeta_{\varepsilon, i} \coloneqq \Gamma_\varepsilon(\widetilde{\zeta}_i)$ for $i = 1, 2$.
Then there exist positive constants $C_3$ and $C_4$ independent of $\widetilde{\zeta}_1$ and $\widetilde{\zeta}_2$ such that
\begin{align}
\label{DP_esti1}
\frac{1}{2} |\zeta_{\varepsilon}(t)|_H^2 
+ \frac{\delta_\sigma}{2} \int^t_0 |\partial_x \zeta_{\varepsilon}(\tau)|_H^2 d\tau
&\leq C_3 \int^t_0 |\widetilde{\zeta}(\tau)|_X^2 d\tau,\\
\label{DP_esti2}
\frac{\delta_\sigma}{2} |\partial_x \zeta_{\varepsilon}(t)|_H^2 
+ \frac{1}{2} \int^t_0 |\partial_t \zeta_{\varepsilon}(\tau)|_H^2 d\tau
&\leq C_4 \int^t_0 |\widetilde{\zeta}(\tau)|_X^2 d\tau
\end{align}
for any $t \in [0, T]$.
\end{lemma}

\begin{proof}
From \eqref{DP_eq_aux}, we see that
\begin{align} \label{DP_eq_aux_diff}
\partial_t \zeta_{\varepsilon} = \sigma_\varepsilon \partial_x^2 \zeta_{\varepsilon} + \varphi_\varepsilon \partial_x \widetilde{\zeta} \qquad \text{in $Q(T)$}.
\end{align}
By multiplying \eqref{DP_eq_aux_diff} by $\zeta_{\varepsilon}$, we have
\begin{align*}
\frac{1}{2} \frac{d}{dt} |\zeta_{\varepsilon}|_H^2 
= \int^1_0 \sigma_\varepsilon \zeta_{\varepsilon} \partial_x^2 \zeta_{\varepsilon} dx
+ \int^1_0 \varphi_\varepsilon \zeta_{\varepsilon} \partial_x \widetilde{\zeta} dx \qquad \text{a.e. on $[0, T]$}.
\end{align*}
Applying integration by parts and Young's inequality to the first term on the right-hand side implies
\begin{align*}
\int^1_0 \sigma_\varepsilon \zeta_{\varepsilon} \partial_x^2 \zeta_{\varepsilon} dx
&= -\int^1_0 (\partial_x \sigma_{\varepsilon}) \zeta_{\varepsilon} (\partial_x \zeta_{\varepsilon}) dx 
- \int^1_0 \sigma_\varepsilon |\partial_x \zeta_{\varepsilon}|^2 dx\\
&\leq \widehat{C}_{\sigma} |\zeta_\varepsilon|_H |\partial_x \zeta_{\varepsilon}|_H
- \delta_\sigma \int^1_0 |\partial_x \zeta_{\varepsilon}|^2 dx\\
&\leq -\frac{\delta_\sigma}{2} |\partial_x \zeta_{\varepsilon}|_H^2 
+ \frac{\widehat{C}_{\sigma}^2}{2\delta_\sigma} |\zeta_{\varepsilon}|_H^2
\end{align*}
a.e. on $[0, T]$, where $\widehat{C}_{\sigma} \coloneqq |\partial_x \sigma_{\varepsilon}|_{L^\infty(Q(T))}$ may depend on $\varepsilon$.
For the second term, thanks to Young's inequality, we obtain
\begin{align*}
\int^1_0 \varphi_\varepsilon \zeta_{\varepsilon} \partial_x \widetilde{\zeta} dx 
&\leq C_1|\zeta_\varepsilon|_H^2 + \frac{1}{2}|\widetilde{\zeta}|_X^2
\qquad \text{a.e. on $[0, T]$},
\end{align*}
where $C_1 \coloneqq |\varphi_\varepsilon|^2_{L^\infty(Q(T))}/2$.
Thus, putting $C_2 \coloneqq (\widehat{C}_{\sigma}^2/(2\delta_\sigma)) + C_1 + (1/2)$, we get
\begin{align*}
\frac{1}{2} \frac{d}{dt} |\zeta_{\varepsilon}|_H^2 
+ \frac{\delta_\sigma}{2} |\partial_x \zeta_{\varepsilon}|_H^2
&\leq C_2(|\zeta_\varepsilon|_H^2 + |\widetilde{\zeta}|_X^2)
\qquad \text{a.e. on $[0, T]$}.
\end{align*}
It follows from Gronwall's inequality that there exists a positive constant $C_3$ such that
\begin{align*}
\frac{1}{2} |\zeta_{\varepsilon}(t)|_H^2 
+ \frac{\delta_\sigma}{2} \int^t_0 |\partial_x \zeta_{\varepsilon}(\tau)|_H^2 d\tau
\leq C_3 \int^t_0 |\widetilde{\zeta}(\tau)|_X^2 d\tau 
\qquad \text{for any $t \in [0, T]$},
\end{align*}
which completes the proof of \eqref{DP_esti1}.

We next prove \eqref{DP_esti2}.
Multiplying \eqref{DP_eq_aux_diff} by $\partial_t \zeta_{\varepsilon}$ gives
\begin{align*}
|\partial_t \zeta_{\varepsilon}|_H^2
&= \int^1_0 \sigma_\varepsilon \partial_t \zeta_\varepsilon \partial_x^2 \zeta_\varepsilon dx
+ \int^1_0 \varphi_\varepsilon \partial_t \zeta_\varepsilon \partial_x \widetilde{\zeta} dx\\
&\leq -\frac{1}{2} \frac{d}{dt} \int^1_0 \sigma_\varepsilon |\partial_x \zeta_{\varepsilon}|^2 dx
+ \frac{1}{2}\int^1_0 |\partial_t \sigma_\varepsilon| |\partial_x \zeta_\varepsilon|^2 dx\\
&\hspace{10mm} + \widehat{C}_\sigma^2 |\partial_x \zeta_\varepsilon|_H^2
+ \frac{1}{2} |\partial_t \zeta_\varepsilon|_H^2
+ 2C_1 |\partial_x \widetilde{\zeta}|_H^2
\end{align*}
a.e. on $[0, T]$.
Therefore, we obtain
\begin{align*}
\frac{1}{2}|\partial_t \zeta_\varepsilon|_H^2 
+ \frac{1}{2} \frac{d}{dt} \int^1_0 \sigma_\varepsilon |\partial_x \zeta_\varepsilon|^2 dx
\leq 2 C_1 |\widetilde{\zeta}|_X^2 
+ \left( \frac{\widehat{C}_\sigma'}{2\delta_\sigma} + \frac{\widehat{C}_\sigma^2}{\delta_\sigma} \right) \int^1_0 \sigma_\varepsilon |\partial_x \zeta_\varepsilon|^2 dx
\end{align*}
a.e. on $[0, T]$, where $\widehat{C}_\sigma' \coloneqq |\partial_t \sigma_\varepsilon|_{L^\infty(Q(T))}^2$.
By Gronwall's inequality, there exists a positive constant $C_4$ such that
\begin{align*}
\frac{\delta_\sigma}{2} |\partial_x \zeta_{\varepsilon}(t)|_H^2
+ \frac{1}{2} \int^t_0 |\partial_t \zeta_{\varepsilon}(\tau)|_H^2 d\tau
\leq C_4 \int^t_0 |\widetilde{\zeta}(\tau)|_X^2 d\tau
\qquad \text{for any $t \in [0, T]$}.
\end{align*}
\end{proof}

We next establish the existence and uniqueness of strong solutions to (DP)($\varepsilon, \xi$).

\begin{proposition} \label{DPep_ss}
Let $\varepsilon > 0$.
If $\xi \in C^\infty_0 (Q(T))$, then {\rm (DP)($\varepsilon, \xi$)} has a unique strong solution $\zeta_\varepsilon \in W^{1,2} (0,T; H) \cap L^\infty (0,T; X) \cap L^2(0,T; H^2(0,1))$.
\end{proposition}

\begin{proof}
Let $\varepsilon > 0$, $\xi \in C^\infty_0 (Q(T))$, and $\widetilde{\zeta}_i \in L^2(0,T; X)$ with $i = 1, 2$.
We put $\widetilde{\zeta} \coloneqq \widetilde{\zeta}_1 - \widetilde{\zeta}_2$ and $\zeta_{\varepsilon} \coloneqq \zeta_{\varepsilon, 1} - \zeta_{\varepsilon, 2}$, where $\zeta_{\varepsilon, i} \coloneqq \Gamma_\varepsilon(\widetilde{\zeta}_i)$ for $i = 1, 2$.
Due to \eqref{DP_esti1} and \eqref{DP_esti2} in Lemma \ref{DP_lem_1},  we have
\begin{align}
\int^{t_1}_0 |(\Gamma_{\varepsilon} (\widetilde{\zeta}_{1})(t) - \Gamma_{\varepsilon} (\widetilde{\zeta}_{2})(t)|_X^2 dt
&= \int^{t_1}_0 |\zeta_\varepsilon(t)|_X^2 dt \nonumber\\
\label{DP_esti3-1}
&\leq \int^{t_1}_0 C_5 \int^t_0 |\widetilde{\zeta}(\tau)|_X^2 d\tau dt \\
\label{DP_esti3-2}
&\leq C_5 t_1 \int^{t_1}_0 |\widetilde{\zeta}(t)|_X^2 dt
\end{align}
for any $t_1 \in [0, T]$, where $C_5 \coloneqq 2C_3 + 2 C_4/\delta_\sigma$.
Taking into account the fact that $\Gamma_{\varepsilon} (\widetilde{\zeta}_i) \in L^2(0,T; X)$, we define $\Gamma^2_{\varepsilon} (\widetilde{\zeta}_i) \coloneqq \Gamma_{\varepsilon} (\Gamma_{\varepsilon} (\widetilde{\zeta}_i))$ for $i = 1, 2$. By using \eqref{DP_esti3-1} and \eqref{DP_esti3-2}, we obtain
\begin{align}
\int^{t_1}_0 |(\Gamma^2_{\varepsilon} (\widetilde{\zeta}_{1})(t) - \Gamma^2_{\varepsilon}(\widetilde{\zeta}_{2})(t)|_X^2 dt
&= \int^{t_1}_0 |(\Gamma_{\varepsilon}(\zeta_{\varepsilon, 1})(t) - \Gamma_{\varepsilon}(\zeta_{\varepsilon, 2})(t)|_X^2 dt \nonumber\\
&\leq \int^{t_1}_0 C_5 \int^{t}_0 |\zeta_\varepsilon(\tau)|_X^2 d\tau dt \nonumber\\
&\leq C_5 \int^{t_1}_0 (C_5 t) \int^t_0 |\widetilde{\zeta}(\tau)|_X^2 d\tau dt \nonumber\\
\label{DP_esti4} 
&\leq \frac{C_5^2 T^2}{2} \int^{t_1}_0 |\widetilde{\zeta}(t)|_X^2 dt
\end{align}
for any $t_1 \in [0, T]$.
We next assume that the following estimate holds for some $k \in \mathbb{N}$:
\begin{align}
\label{DP_esti6}
\int^{t_1}_0 |(\Gamma^{k}(\widetilde{\zeta}_{1})(t) - \Gamma^{k}(\widetilde{\zeta}_{2})(t)|_X^2 dt
\leq \frac{C_5^{k} T^{k}}{k!} \int^{t_1}_0 |\widetilde{\zeta}(t)|_X^2 dt
\qquad \text{for any $t_1 \in [0, T]$}.
\end{align}
In the same way as in the proof of \eqref{DP_esti4}, we can derive
\begin{align*}
\int^{t_1}_0 |(\Gamma^{k+1}(\widetilde{\zeta}_{1})(t) - \Gamma^{k+1}(\widetilde{\zeta}_{2})(t)|_X^2 dt
\leq \frac{C_5^{k+1} T^{k+1}}{(k+1)!} \int^{t_1}_0 |\widetilde{\zeta}(t)|_X^2 dt
\qquad \text{for any $t_1 \in [0, T]$}.
\end{align*}
Therefore, it follows from the induction argument that \eqref{DP_esti6} holds for all $k \in \mathbb{N}$.
Since $(C_5T)^{\widehat{k}}/\widehat{k}! < 1$ for some $\widehat{k} \in \mathbb{N}$, the Banach fixed point theorem implies that there exists $\zeta_{\varepsilon} \in L^2(0, T; H^1(0,1))$ such that $\Gamma_\varepsilon (\zeta_{\varepsilon}) = \zeta_{\varepsilon}$, which is a strong solution to (DP)($\varepsilon, \xi$).

The uniqueness of strong solutions to (DP)($\varepsilon, \xi$) is a direct consequence of \eqref{DP_esti3-2}.
\end{proof}

In Lemmas \ref{DP_lem2} and \ref{DP_lem3} below, we give uniform estimates of $\zeta_\varepsilon$ with respect to $\varepsilon$. 

\begin{lemma} \label{DP_lem2}
Let $\varepsilon > 0$ and let $\xi \in C_0^\infty (Q(T))$.
Let $\zeta_\varepsilon$ be the strong solution of {\rm (DP)($\varepsilon, \zeta$)} given in Proposition \ref{DPep_ss}.
Then, there exists a positive constant $M > 0$ independent of $\varepsilon$ such that
\begin{align}
\label{zeta_esti}|\zeta_\varepsilon(t, x)| \leq Mt 
\qquad \text{for $(t, x) \in Q(T)$}.
\end{align}
\end{lemma}

\begin{proof}
Let $( \cdot )^+$ denote the positive part and let $M > \max_{\overline{Q(T)}} |\xi|$.
Multiplying \eqref{DP_eq} by $(\zeta_\varepsilon(t) - Mt)^+$ yields
\begin{align*}
\int^1_0 \xi(t) (\zeta_\varepsilon(t) - Mt)^+ dx
&= \int^1_0 (\zeta_\varepsilon(t) - Mt)^+ \partial_t \zeta_{\varepsilon}(t) dx 
-\int^1_0 \sigma_\varepsilon \partial_x^2 \zeta_{\varepsilon}(t) (\zeta_\varepsilon(t) - Mt)^+ dx \\
&\hspace{10mm} - \int^1_0 \varphi_\varepsilon(t) \partial_x \zeta_{\varepsilon}(t) (\zeta_\varepsilon(t) - Mt)^+ dx\\
&\eqqcolon I_1(t) + I_2(t) + I_3(t)
\end{align*} 
for a.e. $t \in [0, T]$.
By a simple calculation, we get
\begin{align*}
I_1(t)
= \frac{1}{2} \frac{d}{dt} |(\zeta_\varepsilon(t) - Mt)^+|_H^2 
+ \int^1_0 M(\zeta_\varepsilon(t) - Mt)^+ dx
\qquad \text{for a.e. $t \in [0, T]$}.
\end{align*}
Applying integration by parts to $I_2$, together with \eqref{sigma_esti}, we have
\begin{align*}
&I_2(t)
\geq \delta_\sigma |\partial_x(\zeta_\varepsilon(t) - Mt)^+|_H^2
+ \int^1_0 (\zeta_\varepsilon(t) - Mt)^+ \partial_x(\zeta_\varepsilon(t) - Mt)^+ \partial_x\sigma_\varepsilon(t) dx\\
&\hspace{110mm} \text{for a.e. $t \in [0, T]$}.
\end{align*}
It follows from Young's inequality that
\begin{align*}
I_3(t)
&\geq -|\varphi_\varepsilon|_{L^\infty(Q(T))} \int^1_0 |\partial_x \zeta_\varepsilon(t)| |(\zeta_\varepsilon(t) - Mt)^+| dx\\
&\geq -|\varphi_\varepsilon|_{L^\infty(Q(T))} \int^1_0 |\partial_x (\zeta_\varepsilon(t) - Mt)^+| |(\zeta_\varepsilon(t) - Mt)^+| dx\\
&\geq -\frac{\delta_\sigma}{2} |\partial_x(\zeta_\varepsilon(t) - Mt)^+|_H^2
- \frac{1}{2\delta_\sigma} |\varphi_\varepsilon|_{L^\infty(Q(T))}^2 |(\zeta_\varepsilon(t) - Mt)^+|_H^2
\end{align*}
for a.e. $t \in [0, T]$.
Since $M > \max_{\overline{Q(T)}} |\xi|$, combining these estimates gives
\begin{align*}
&\frac{1}{2} \frac{d}{dt} |(\zeta_\varepsilon(t) - Mt)^+|_H^2 
+ \frac{\delta_\sigma}{2} |\partial_x(\zeta_\varepsilon(t) - Mt)^+|_H^2\\
&\leq
- \int^1_0 (\zeta_\varepsilon(t) - Mt)^+ \partial_x(\zeta_\varepsilon(t) - Mt)^+ \partial_x\sigma_\varepsilon(t) dx
+ \int^1_0 (\xi(t) - M)(\zeta_\varepsilon(t) - Mt)^+ dx\\
&\hspace{10mm} +\frac{1}{2 \delta_\sigma} |\varphi_\varepsilon|_{L^\infty(Q(T))}^2 |(\zeta_\varepsilon(t) - Mt)^+|_H^2\\
&\leq \widehat{C}_{\sigma} \int^1_0 |(\zeta_\varepsilon(t) - Mt)^+| | \partial_x(\zeta_\varepsilon(t) - Mt)^+| dx
 +\frac{1}{2 \delta_\sigma} |\varphi_\varepsilon|_{L^\infty(Q(T))}^2 |(\zeta_\varepsilon(t) - Mt)^+|_H^2
\end{align*}
for a.e. $t \in [0, T]$.
By setting $C_{6} \coloneqq \widehat{C}_{\sigma}^2/\delta_\sigma + C_q/(2 \delta_\sigma)$ and using Young's inequality, we obtain
\begin{align*}
&\frac{1}{2} \frac{d}{dt} |(\zeta_\varepsilon(t) - Mt)^+|_H^2 
+ \frac{\delta_\sigma}{4} | \partial_x(\zeta_\varepsilon(t) - Mt)^+|_H^2 \\
&\hspace{10mm} \leq C_{6} |(\zeta_\varepsilon(t) - Mt)^+|_H^2
\qquad \text{for a.e. $t \in [0, T]$}.
\end{align*}
From Gronwall's inequality and \eqref{DP_IC}, we conclude that $|(\zeta_\varepsilon(t) - Mt)^+|_H^2 = 0$ for any $t \in [0, T]$, which implies
\begin{align*}
\zeta_\varepsilon(t, x) \leq Mt
\qquad \text{for $(t, x) \in Q(T)$}.
\end{align*}
Similarly, multiplying \eqref{DP_eq_aux} by $(-\zeta_\varepsilon - Mt)^+$ instead of $(\zeta_\varepsilon - Mt)^+$, we can derive
\begin{align*}
-\zeta_\varepsilon(t, x) \leq Mt
\qquad \text{for $(t, x) \in Q(T)$}.
\end{align*}
\end{proof}

\begin{lemma} \label{DP_lem3}
Under the same assumption as in Lemma \ref{DP_lem2}, there exists a positive constant $C_8$ independent of $\varepsilon$ such that the following estimate holds:
\begin{align*}
\frac{1}{2} |\partial_x \zeta_{\varepsilon}(t_1)|_H^2
+ \frac{\delta_\sigma}{4} \int^T_0 |\partial_x^2 \zeta_{\varepsilon}(t)|_H^2 dt
\leq C_8
\qquad \text{for $0 \leq t_1 \leq T$}.
\end{align*}
\end{lemma}

\begin{proof}
Multiplying \eqref{DP_eq} by $\partial_x^2 \zeta_{\varepsilon}$ gives
\begin{align*}
-\int^1_0 \partial_t \zeta_{\varepsilon} \partial_x^2 \zeta_{\varepsilon} dx 
+\int^1_0 \sigma_\varepsilon (\partial_x^2 \zeta_{\varepsilon})^2 dx 
= - \int^1_0 \xi\partial_x^2\zeta_{\varepsilon} dx
- \int^1_0 \varphi_\varepsilon \partial_x \zeta_{\varepsilon} \partial_x^2 \zeta_{\varepsilon} dx 
\qquad \text{a.e. on $[0, T]$}.
\end{align*}
By integration by parts and \eqref{zeta_esti}, we obtain
\begin{align*}
-\int^1_0 \partial_t \zeta_{\varepsilon}\partial_x^2\zeta_{\varepsilon} dx
&= \frac{1}{2} \frac{d}{dt} |\partial_x \zeta_{\varepsilon}|_H^2,\\
- \int^1_0 \xi \partial_x^2 \zeta_{\varepsilon} dx
&= -\int^1_0 (\partial_x^2 \xi) \zeta_{\varepsilon} dx\\
&\leq MT \int^1_0 |\partial_x^2 \xi| dx
\end{align*}
a.e. on $[0, T]$.
Taking into account the fact that $\xi \in C_0^\infty(Q(T))$, we put
\begin{align*}
C_7 \coloneqq MT \sup_{0 \leq t \leq T}\int^1_0 |\partial_x^2 \xi(t)| dx.
\end{align*}
By using \eqref{sigma_esti}, we have
\begin{align*}
\int^1_0 \sigma_\varepsilon |\partial_x^2 \zeta_{\varepsilon}|^2 dx
\geq \delta_\sigma |\partial_x^2 \zeta_{\varepsilon}|_H^2
\qquad \text{a.e. on $[0, T]$}.
\end{align*}
Furthermore, it follows from \eqref{1DGNiineq_Linf} and Young's inequality that
\begin{align*}
- \int^1_0 \varphi_\varepsilon \partial_x \zeta_{\varepsilon}\partial_x^2 \zeta_{\varepsilon} dx
&\leq \frac{\delta_\sigma}{2} |\partial_x^2 \zeta_{\varepsilon}|_H^2
+ \frac{1}{2\delta_\sigma} \int^1_0 |\varphi_\varepsilon|^2 |\partial_x \zeta_{\varepsilon}|^2 dx\\
&\leq \frac{\delta_\sigma}{2} |\partial_x^2 \zeta_{\varepsilon}|_H^2
+ \frac{1}{2\delta_\sigma} (|\partial_x \zeta_{\varepsilon}|_H^2 + 2|\partial_x \zeta_{\varepsilon}|_H|\partial_x^2 \zeta_{\varepsilon}|_H)|\varphi_\varepsilon|_H^2\\
&\leq \frac{3}{4} \delta_\sigma |\partial_x^2 \zeta_{\varepsilon}|_H^2
+ \frac{1}{2 \delta_\sigma} |\partial_x \zeta_{\varepsilon}|_H^2 |\varphi_\varepsilon|_H^2
+ \frac{1}{\delta_\sigma} |\partial_x \zeta_{\varepsilon}|_H^2 |\varphi_\varepsilon|_H^4
\end{align*}
a.e. on $[0, T]$.
Thus, we get
\begin{align*}
\frac{1}{2} \frac{d}{dt} |\partial_x \zeta_{\varepsilon}|_H^2
+ \frac{\delta_\sigma}{4} |\partial_x^2 \zeta_{\varepsilon}|_H^2
\leq \frac{1}{2 \delta_\sigma} |\partial_x \zeta_{\varepsilon}|_H^2 |\varphi_\varepsilon|_H^2
+ \frac{1}{\delta_\sigma} |\partial_x \zeta_{\varepsilon}|_H^2 |\varphi_\varepsilon|_H^4
+ C_7
\qquad \text{a.e. on $[0, T]$}.
\end{align*}
By using Gronwall's inequality and \eqref{varphi_esti}, we obtain
\begin{align}
&\frac{1}{2} |\partial_x \zeta_{\varepsilon}(t_1)|_H^2
+ \frac{\delta_\sigma}{4} \int^{t_1}_0 |\partial_x^2 \zeta_{\varepsilon}(t)|_H^2 dt \nonumber\\
&\leq C_7t_1 \exp \bigg( \frac{1}{2\delta_\sigma} \int^T_0 |\varphi_\varepsilon(t)|_H^2 dt
+ \frac{1}{\delta_\sigma} \int^T_0 |\varphi_\varepsilon(t)|_H^4 dt \bigg)  \nonumber\\
&\label{DP_esti7}\leq C_7T \exp \bigg( \frac{1}{2\delta_\sigma} C_q^2 \int^T_0 |p|_H^2 dt
+ \frac{C_q^4}{\delta_\sigma} \int^T_0 |p|_H^4 dt \bigg)\\
&\leq C_7T \exp \bigg( \frac{C_q^2}{4\delta_\sigma}(|p|_{L^4(0,T; H)}^4 + 1) 
+ \frac{C_q^4}{\delta_\sigma} |p|_{L^4(0,T; H)}^4 \bigg)  \nonumber \\
& \eqqcolon C_8 \nonumber
\end{align}
for $0 \leq t_1 \leq T$.
This completes the proof of Lemma \ref{DP_lem3}.
\end{proof}

\begin{remark}
As seen from \eqref{DP_esti7}, the assumption $p \in L^4(0, T; H)$ in Theorem \ref{main_thm} is required for the uniqueness of weak solutions to (P)($p, v_0$).
\end{remark}

\begin{proof}[Proof of Threorem \ref{main_thm}]
Let $v_1$ and $v_2$ be weak solutions of (P)($p, v_0$).
We set $z_1 \coloneqq h(v_1)$, $z_2 \coloneqq h(v_2)$, and $z \coloneqq z_1 - z_2$.
In addition, for $\varepsilon > 0$ and $\xi \in C_0^\infty(Q(T))$, let $\zeta_\varepsilon$ be the strong solutions of (DP)$(\varepsilon, \xi)$ given in Proposition \ref{DPep_ss}.
By \eqref{B2_diff} and \eqref{DP_eq}, we have
\begin{align*}
\bigg| \int^T_0 \int^1_0 z \xi dx dt \bigg|
= \bigg| \int^T_0 \int^1_0 z((\sigma_\varepsilon - \sigma)\partial_x^2 \zeta_{\varepsilon} + (\varphi_\varepsilon - \varphi)\partial_x \zeta_{\varepsilon}) dx dt \bigg|.
\end{align*}
It follows from \eqref{P_D1} that $z \in L^4(Q(T))$. Thanks to \eqref{1DGNiineq_Linf} and Lemma \ref{DP_lem3}, we observe that $\partial_x \zeta_\varepsilon$ is uniformly bounded in $L^4(Q(T))$ with respect to $\varepsilon$.
Indeed, we obtain
\begin{align*}
|\partial_x \zeta_\varepsilon|_{L^4(Q(T))}^4
&\leq \int^T_0 |\partial_x \zeta_\varepsilon|_{L^\infty(0,1)}^2 \int_0^1 |\partial_x \zeta_\varepsilon|^2 dx dt\\
&\leq \int^T_0 (|\partial_x \zeta_\varepsilon|_H^2 + 2|\partial_x \zeta_\varepsilon|_H|\partial_x^2 \zeta_\varepsilon|_H)|\partial_x \zeta_\varepsilon|_H^2 dt \\
&\leq \int^T_0 (2|\partial_x \zeta_\varepsilon|_H^2 + |\partial_x^2 \zeta_\varepsilon|_H^2)|\partial_x \zeta_\varepsilon|_H^2 dt\\
&\leq 2 C_8^2 T + \frac{4C_8^2}{\delta_\sigma}
\end{align*}
for $\varepsilon > 0$.
Thus, for any $\xi \in C_0^\infty(Q(T))$, we get
\begin{align*}
&\bigg| \int^T_0 \int^1_0 z \xi dx dt \bigg|\\
&\leq \int^T_0 \int^1_0 |z| |\sigma_\varepsilon - \sigma| |\partial_x^2 \zeta_{\varepsilon}| dx dt
+ \int^T_0 \int^1_0 |z| |\varphi_\varepsilon - \varphi| |\partial_x \zeta_{\varepsilon}| dx dt\\
&\leq |z|_{L^4(Q(T))} |\sigma_\varepsilon - \sigma|_{L^4(Q(T))} |\partial_x \zeta_\varepsilon|_{L^2(Q(T))}
+ |z|_{L^4(Q(T))} |\varphi_\varepsilon - \varphi|_{L^2(Q(T))} |\partial_x \zeta_{\varepsilon}|_{L^4(Q(T))}\\
&\to 0 
\end{align*}
as $\varepsilon \to 0$.
This implies $z = 0$, that is, $v_1 = v_2$ on $Q(T)$.
\end{proof}

\section{Convergence of approximate solutions} \label{AS_proof}
In this section, we give the proof of Theorem \ref{AS_conv}.
For this purpose, we suppose the same assumption as in Proposition \ref{P_e} in the rest of this paper.
Let $\delta > 0$ and let $n \in \mathbb{N}$.
We recall that $v_{\delta}^{(n)}$ is the function on $Q(T)$ defined by \eqref{NS}.
In addition, we set
\begin{align*}
v_{0}^{(n)}(x) \coloneqq \sum^n_{i = 1} v_{0, i}^{(n)} \chi_i^{(n)}(x)
\qquad \text{for $x \in [0, 1]$}.
\end{align*}

By simple computations, we obtain
\begin{alignat}{2}
\label{AS_IC_esti}
&|v_{0}^{(n)}|_H \leq |v_0|_H 
&&\qquad \text{for each $n \in \mathbb{N}$}, \\
&v_{0}^{(n)} \to v_0 
&&\qquad \text{in $H$ as $n \to \infty$}. \nonumber
\end{alignat}
Moreover, it follows from (A1) that there exist positive constants $C_9, \ldots ,C_{14}$ such that
\begin{align}
\label{AS_h_esti1}
r^2 &\leq C_9 |h(r)|^2 + C_{10} \leq C_{11} \widehat{h}(h(r)) + C_{12}, \\
\label{AS_h_esti2}
\widehat{h}(r) &\leq C_{13} |r|^2 + C_{14}
\end{align}
for any $r \in \mathbb{R}$, where 
\begin{align*}
\widehat{h}(r) \coloneqq \int^r_0 h^{-1} (\xi) d\xi.
\end{align*}
In fact, if $r \in \mathbb{R}$ satisfies $h(r) \geq 0$, then we have 
\begin{align}
\widehat{h}(h(r))
&= \int^{h(r)}_0 (h^{-1}(\xi) - h^{-1}(0)) d\xi + \int^{h(r)}_0 h^{-1}(0) d\xi \nonumber\\
&\geq \int^{h(r)}_0 \int^\xi_0 (h^{-1})'(y) dy d\xi - \left| \int^{h(r)}_0 h^{-1}(0) d\xi \right| \nonumber\\
&\geq \frac{1}{C_h} \int^{h(r)}_0 \int^\xi_0 dy d\xi - |h^{-1}(0)| |h(r)| \nonumber\\
\label{AS_h_esti3}
&\geq \frac{1}{4C_h} |h(r)|^2 - C_h |h^{-1}(0)|^2.
\end{align}
Similarly, \eqref{AS_h_esti3} can be derived even for $r \in \mathbb{R}$ with $h(r) < 0$.
On the other hand, for any $r > 0$, the mean value theorem implies that there exists $y \in (0, r)$ such that
\begin{align*}
h(r) - h(0) = h'(y) r.
\end{align*}
With the aid of this identity, we get
\begin{align}
\label{AS_h_esti4} r^2 \leq \frac{2}{\delta_h^2} |h(r)|^2 + \frac{2}{\delta_h^2}|h(0)|^2.
\end{align}
In the same way, we can show that the above inequality holds for all $r < 0$.
As a consequence, combining \eqref{AS_h_esti3} with \eqref{AS_h_esti4} yields \eqref{AS_h_esti1}.
We can prove \eqref{AS_h_esti2} by a straightforward computation.

\begin{lemma}\label{AS_lem_esti1}
For $\delta > 0$ and $n \in \mathbb{N}$, let $v_{\delta}^{(n)}$ be the function on $Q(T)$ defined by \eqref{NS} and put $z_{\delta}^{(n)} \coloneqq h(v_{\delta}^{(n)})$.
Then, there exists a positive constant $C_{15}$ independent of $n$ and $\delta$ such that
\begin{alignat}{3}
\label{AS_esti1}
\int^1_0 |v_{\delta}^{(n)}(t)|^2 dx &\leq C_{15},
&\qquad \int^1_0 |z_{\delta}^{(n)}(t)|^2 dx &\leq C_{15}
&&\qquad \text{for $0 \leq t \leq T$},\\
\label{AS_esti4}
\int^T_0 \int^1_0 | \widetilde{v}_{\delta}^{(n)} |^2 dx dt &\leq C_{15}, 
&\qquad \int^T_0 \int^1_0 | \widetilde{z}_{\delta}^{(n)} |^2 dx dt &\leq C_{15},
&&
\end{alignat}
where
\begin{align*}
\widetilde{v}_{\delta}^{(n)} (t, x) &\coloneqq
\begin{dcases}
0, &\qquad 0 \leq x < \Delta x^{(n)},\\
\frac{v_{\delta}^{(n)}(t, x) - v_{\delta}^{(n)}(t, x - \Delta x^{(n)})}{\Delta x^{(n)}}, &\qquad \Delta x^{(n)} \leq x \leq 1,
\end{dcases}\\
\widetilde{z}_{\delta}^{(n)} (t, x) &\coloneqq
\begin{dcases}
0, &\qquad 0 \leq x < \Delta x^{(n)},\\
\frac{z_{\delta}^{(n)}(t, x) - z_{\delta}^{(n)}(t, x - \Delta x^{(n)})}{\Delta x^{(n)}}, &\qquad \Delta x^{(n)} \leq x \leq 1.
\end{dcases}
\end{align*}
\end{lemma}

\begin{proof}
Let $\delta > 0$ and let $n \in \mathbb{N}$.
Let $x_i^{(n)} = i\Delta x^{(n)}$ for $i = 0, 1, 2, \ldots, n$.
From \eqref{OP_eq}, we get
\begin{align}
&\int^1_0 h' (v_{\delta}^{(n)}) (\partial_t v_{\delta}^{(n)}) v_{\delta}^{(n)} dx \nonumber \\
&= \sum^{n-1}_{i = 1}  v_{\delta, i}^{(n)} \left( \frac{v_{\delta, i+1}^{(n)} - v_{\delta, i}^{(n)}}{\Delta x^{(n)}} + b\left(\frac{v_{\delta, i}^{(n)} + v_{\delta, i+1}^{(n)}}{2}\right) \rho_{\delta, i}^{(n)} \right) \nonumber \\
\label{AS_esti8}
&\hspace{10mm} - \sum^n_{i = 2} v_{\delta, i}^{(n)} \left( \frac{v_{\delta, i}^{(n)} - v_{\delta, i-1}^{(n)}}{\Delta x^{(n)}} + b \left( \frac{v_{\delta, i-1}^{(n)} + v_{\delta, i}^{(n)}}{2} \right) \rho_{\delta, i-1}^{(n)} \right)
\qquad \text{a.e. on $[0, T]$}.
\end{align}
For the left-hand side, we have
\begin{align*}
\int^1_0 h' (v_{\delta}^{(n)}) (\partial_t v_{\delta}^{(n)}) v_{\delta}^{(n)} dx 
&=\int^1_0 \partial_t h(v_{\delta}^{(n)}) v_{\delta}^{(n)} dx\\
&= \int^1_0 \widehat{h}' (z_{\delta}^{(n)}) \partial_t z_{\delta}^{(n)} dx\\
&= \frac{d}{dt} \int^1_0 \widehat{h} (z_{\delta}^{(n)}) dx
\end{align*}
a.e. on $[0, T]$.
On the other hand, the right-hand side of \eqref{AS_esti8} is estimated as
\begin{align*}
&\sum^{n-1}_{i = 1}  v_{\delta, i}^{(n)} \left( \frac{v_{\delta, i+1}^{(n)} - v_{\delta, i}^{(n)}}{\Delta x^{(n)}} + b\left(\frac{v_{\delta, i}^{(n)} + v_{\delta, i+1}^{(n)}}{2}\right) \rho_{\delta, i}^{(n)} \right) \\
&\hspace{10mm} - \sum^n_{i = 2} v_{\delta, i}^{(n)} \left( \frac{v_{\delta, i}^{(n)} - v_{\delta, i-1}^{(n)}}{\Delta x^{(n)}} + b \left( \frac{v_{\delta, i-1}^{(n)} + v_{\delta, i}^{(n)}}{2} \right) \rho_{\delta, i-1}^{(n)} \right)\\
&= \sum^n_{i=2} \left( \frac{v_{\delta, i}^{(n)}(t) - v_{\delta, i-1}^{(n)}(t)}{\Delta x^{(n)}} \right) (v_{\delta, i-1}^{(n)}(t) - v_{\delta, i}^{(n)}(t))\\
&\hspace{10mm} + \sum^n_{i=2} b \left( \frac{v_{\delta, i-1}^{(n)}(t) + v_{\delta, i}^{(n)}(t)}{2} \right) \rho_{\delta, i-1}^{(n)}(t) (v_{\delta, i-1}^{(n)}(t) - v_{\delta, i}^{(n)}(t))\\
&= - \sum^n_{i = 2} \int^{x_i^{(n)}}_{x_{i-1}^{(n)}} \left| \frac{v_{\delta, i}^{(n)}(t) - v_{\delta, i-1}^{(n)}(t)}{\Delta x^{(n)}} \right|^2 dx\\ 
&\hspace{10mm} - \sum^n_{i = 2} \int^{x_i^{(n)}}_{x_{i-1}^{(n)}} b \left( \frac{v_{\delta, i-1}^{(n)}(t) + v_{\delta, i}^{(n)}(t)}{2} \right) \rho_{\delta, i-1}^{(n)}(t) \left( \frac{v_{\delta, i}^{(n)}(t) - v_{\delta, i-1}^{(n)}(t)}{\Delta x^{(n)}} \right) dx\\
&\leq - \int^1_{\Delta x^{(n)}} \left| \frac{v_{\delta}^{(n)}(t, x) - v_{\delta}^{(n)}(t, x - \Delta x^{(n)})}{\Delta x^{(n)}} \right|^2 dx\\
&\hspace{10mm} + C_{b} \int^1_{\Delta x^{(n)}} \left| \frac{v_{\delta}^{(n)}(t, x) - v_{\delta}^{(n)}(t, x - \Delta x^{(n)})}{\Delta x^{(n)}} \right| |\rho_{\delta}^{(n)}(t, x-\Delta x^{(n)})| dx
\end{align*} 
for a.e. $t \in [0, T]$, where 
\begin{align*}
\rho_{\delta}^{(n)}(t, x) = \sum^n_{i=1} \chi_i^{(n)}(x) \rho_{\delta, i}^{(n)}(t) 
\qquad \text{for $(t, x) \in Q(T)$}.
\end{align*}
Hence, we obtain
\begin{align*}
&\frac{d}{dt} \int^1_0 \widehat{h} (z_{\delta}^{(n)}(t, x)) dx
+ \frac{1}{2}\int^1_{\Delta x^{(n)}} \left| \frac{v_{\delta}^{(n)}(t, x) - v_{\delta}^{(n)}(t, x - \Delta x^{(n)})}{\Delta x^{(n)}} \right|^2 dx \\
&\leq \frac{C_{b}^2}{2} \int^1_{\Delta x^{(n)}} | \rho_{\delta}^{(n)}(t, x-\Delta x^{(n)})|^2 dx 
\qquad \text{for a.e. $t \in [0, T]$}.
\end{align*}
Gronwall's inequality gives
\begin{align*}
&\int^1_0 \widehat{h} (z_{\delta}^{(n)}(t', x)) dx
+ \frac{1}{2} \int^{t'}_0 \int^1_{\Delta x^{(n)}} \left| \frac{v_{\delta}^{(n)}(t, x) - v_{\delta}^{(n)}(t, x - \Delta x^{(n)})}{\Delta x^{(n)}} \right|^2 dx dt\\
&\leq \frac{C_{b}^2}{2}\int^T_0 \int^1_{\Delta x^{(n)}} | \rho_{\delta}^{(n)}(t, x-\Delta x^{(n)})|^2 dx dt
+ \int^1_0 \widehat{h} (z_{\delta}^{(n)}(0, x)) dx 
\qquad \text{for any $t' \in [0, T]$}.
\end{align*}
It follows from the property of the mollifier and the definition of $\rho_{\delta}^{(n)}$ that
\begin{align}
\int^T_0 \int^1_{\Delta x^{(n)}} | \rho_{\delta}^{(n)}(t, x-\Delta x^{(n)})|^2 dx dt
&= \int^T_0 \sum^{n-1}_{i = 1} \int^{x_i^{(n)}}_{x_{i-1}^{(n)}} | \rho_{\delta, i}^{(n)}(t)|^2 dx dt \nonumber \\
&\leq\int^T_0 \sum^{n-1}_{i = 1} \int^{x_i^{(n)}}_{x_{i-1}^{(n)}} \bigg| \frac{1}{\Delta x^{(n)}} \int^{x_i^{(n)}}_{x_{i-1}^{(n)}} |\rho_\delta(t, \xi)| d\xi \bigg|^2 dx dt \nonumber \\
&\leq \frac{1}{\Delta x^{(n)}} \int^T_0 \sum^{n-1}_{i = 1} \int^{x_i^{(n)}}_{x_{i-1}^{(n)}} \int^{x_i^{(n)}}_{x_{i-1}^{(n)}} |\rho_\delta(t, \xi)|^2 d\xi dx dt \nonumber \\
&= \int^T_0 \sum^{n-1}_{i = 1} \int^{x_i^{(n)}}_{x_{i-1}^{(n)}} |\rho_\delta(t, x)|^2 dx dt \nonumber \\
&\leq \int^T_0 |\rho_\delta(t, \cdot)|_H^2 dt \nonumber \\
&\label{AS_esti5}\leq \int^T_0 |p(t, \cdot)|_H^2 dt.
\end{align}
Thanks to $z_{\delta}^{(n)}(0) = h(v_{0}^{(n)})$, \eqref{AS_IC_esti}, and \eqref{AS_h_esti2}, we get
\begin{align*}
\int^1_0 \widehat{h} (z_{\delta}^{(n)}(0, x)) dx
&\leq C_{13} \int^1_0 |z_{\delta}^{(n)}(0, \cdot)|^2 dx + C_{14}\\
&\leq C_{13} \int^1_0 |h(v_0)|^2 dx + C_{14}.
\end{align*}
As a consequence, we arrive at \eqref{AS_esti1} and \eqref{AS_esti4} with the aid of \eqref{AS_h_esti1}.
\end{proof}

\begin{lemma} \label{AS_lem_esti2}
There exists a positive constant $C_{17}$ independent of $n$ and $\delta$ such that
\begin{align} \label{AS_esti7}
\left| \int^T_0 \int^1_0 (\partial_t z_{\delta}^{(n)}) \eta dx \right| \leq C_{17} |\eta|_{L^2(0,T; X)}
\qquad \text{for any $\eta \in L^2(0,T; X)$}.
\end{align}
\end{lemma}

\begin{proof}
Let $n \in \mathbb{N}$ and let $\delta > 0$.
We put
\begin{align*}
a_{\delta, i}^{(n)} \coloneqq \frac{v_{\delta, i+1}^{(n)} - v_{\delta, i}^{(n)}}{\Delta x^{(n)}} + b\left(\frac{v_{\delta, i}^{(n)} + v_{\delta, i+1}^{(n)}}{2}\right) \rho^{(n)}_{\delta, i} 
\qquad \text{for $i = 1, 2, \ldots, n-1$}.
\end{align*}
Multiplying \eqref{OP_eq} by $\eta \in L^2(0,T; H)$ and integrating both sides over $Q(T)$, we obtain
\begin{align*}
\left| \int^T_0 \int^1_0 (\partial_t z_{\delta}^{(n)}) \eta dx \right|
= \left| \int^T_0 \left( \frac{1}{\Delta x^{(n)}} \sum^{n-1}_{i = 1} \int^{x_i^{(n)}}_{x_{i-1}^{(n)}} a_{\delta, i}^{(n)} \eta dx - \frac{1}{\Delta x^{(n)}} \sum^{n}_{i = 2} \int^{x_{i}^{(n)}}_{x_{i-1}^{(n)}} a_{\delta, i-1}^{(n)} \eta dx \right) dt \right|.
\end{align*}
We note that
\begin{align*}
\frac{1}{\Delta x^{(n)}} \sum^{n}_{i = 2} \int^{x_{i}^{(n)}}_{x_{i-1}^{(n)}} a_{\delta, i-1}^{(n)}(t) \eta(t, x) dx
&= \frac{1}{\Delta x^{(n)}} \sum^{n-1}_{i = 1} a_{\delta, i}^{(n)}(t) \int^{x_{i+1}^{(n)}}_{x_i^{(n)}} \eta(t, x) dx\\
&= \frac{1}{\Delta x^{(n)}} \sum^{n-1}_{i = 1} a_{\delta, i}^{(n)}(t) \int^{x_i^{(n)}}_{x_{i-1}^{(n)}} \eta(t, x + \Delta x^{(n)}) dx
\end{align*}
for any $t \in [0, T]$.
Thus, we have
\begin{align}
&\int^T_0 \left( \frac{1}{\Delta x^{(n)}} \sum^{n-1}_{i = 1} \int^{x_i^{(n)}}_{x_{i-1}^{(n)}} a_{\delta, i}^{(n)} \eta dx - \frac{1}{\Delta x^{(n)}} \sum^{n}_{i = 2} \int^{x_{i}^{(n)}}_{x_{i-1}^{(n)}} a_{\delta, i-1}^{(n)} \eta dx \right) dt \nonumber\\
&= \int^T_0 \sum^{n-1}_{i = 1} \int^{x_i^{(n)}}_{x_{i-1}^{(n)}} a_{\delta, i}^{(n)}(t) \frac{\eta(t, x) - \eta(t, x + \Delta x^{(n)})}{\Delta x^{(n)}} dx dt \nonumber \\
&= \int^T_0 \sum^{n-1}_{i = 1} \int^{x_i^{(n)}}_{x_{i-1}^{(n)}} \frac{v_{\delta, i+1}^{(n)}(t) - v_{\delta, i}^{(n)}(t)}{\Delta x^{(n)}}  \frac{\eta(t, x) - \eta(t, x + \Delta x^{(n)})}{\Delta x^{(n)}} dx dt \nonumber \\
\label{AS_esti6}
&\hspace{10mm} + \int^T_0 \sum^{n-1}_{i = 1} \int^{x_i^{(n)}}_{x_{i-1}^{(n)}} b\left(\frac{v_{\delta, i}^{(n)}(t) + v_{\delta, i+1}^{(n)}(t)}{2}\right) \rho^{(n)}_{\delta, i}(t) \frac{\eta(t, x) - \eta(t, x + \Delta x^{(n)})}{\Delta x^{(n)}} dx dt\\
&\eqqcolon I_1 + I_2. \nonumber
\end{align}
Applying H\"older's inequality and Lemma \ref{AS_lem_esti1} gives
\begin{align*}
|I_1| 
&\leq \left\{ \int^T_0 \int^{1-\Delta x^{(n)}}_0 \left| \frac{v_{\delta}^{(n)}(t, x) - v_{\delta}^{(n)}(t, x - \Delta x^{(n)})}{\Delta x^{(n)}} \right|^2 dx dt \right\}^{\frac{1}{2}} 
\left\{ \int^T_0 \left| \partial_x \eta (t, \cdot)\right|^2_H dt \right\}^{\frac{1}{2}}\\
&\leq C_{15}  |\eta|_{L^2(0,T; X)}.
\end{align*} 
It follows from (A2) and \eqref{AS_esti5} that
\begin{align*}
|I_2| &\leq C_{b} \left\{ \int^T_0 \int^{1-\Delta x^{(n)}}_0 | \rho_{\delta}^{(n)} (t, x - \Delta x^{(n)}) |^2 dx dt \right\}^{\frac{1}{2}}
\left\{ \int^T_0 \int^1_{\Delta x^{(n)}} \left| \partial_x \eta (t, x)\right|^2 dx dt \right\}^{\frac{1}{2}}\\
&\leq C_{b} |p|_{L^2(0,T; H)}  |\eta|_{L^2(0,T; X)}.
\end{align*}
Hence, by setting $C_{17} \coloneqq C_{15} + |p|_{L^2(0,T; H)}$, we arrive at the desired conclusion.
\end{proof}

We introduce a discrete version of the Aubin--Lions compactness theorem established by Gallou\"{e}t--Latch\'{e} \cite{G-L}.

\begin{lemma}[{\cite[Theorem 3.4]{G-L}}] \label{discrete_Aubin}
For $n \in \mathbb{N}$, define $u^{(n)} \colon Q(T) \to \mathbb{R}$ by
\begin{align*}
u^{(n)}(t, x) \coloneqq \sum^{m_n}_{i = 1} u_i^{(n)}(t) \chi_i^{(m_n)} (x)
\qquad \text{on $Q(T)$},
\end{align*}
where $m_n \in \mathbb{N}$ with $m_n \to \infty$ as $n \to \infty$, $u_i^{(n)} \colon [0, T] \to \mathbb{R}$ is a measurable function for each $i = 1, 2, \ldots, m_n$, $\chi_i^{(m_n)}$ is the charactristic function of $[ (i - 1)\Delta x^{(m_n)}, i\Delta x^{(m_n)})$ for $i = 1, 2, \ldots, m_n$, and $\Delta x^{(m_n)} = 1/m_n$.
In addition, for $t \in [0, T]$, put
\begin{align*}
\widetilde{u}^{(n)} (t, x) \coloneqq 
\begin{dcases}
0, &\qquad 0 \leq x < \Delta x^{(m_n)},\\
\frac{u^{(n)}(t, x) - u^{(n)}(t, x - \Delta x^{(m_n)})}{\Delta x^{(m_n)}}, &\qquad \Delta x^{(m_n)} \leq x \leq 1.
\end{dcases}
\end{align*}
If $\{ u^{(n)} \}$, $\{ \widetilde{u}^{(n)} \}$, and $\{ \partial_t u^{(n)} \}$ are bounded in $L^\infty(0,T; H)$, $L^2(0,T; H)$, and  $L^2(0,T; X^\ast)$, respetctively,
then there exist a subsequence $\{ u^{(n_j)} \}$ of $\{ u^{(n)} \}$ and $u \in L^2(0,T; H)$ such that $u^{(n_j)} \to u$ \ in \ $L^2(0,T; H)$ \ as \ $j \to \infty$, where $X^\ast$ denotes the dual space of $X$.
\end{lemma}

\begin{remark}
The authors in \cite{G-L} established the Aubin--Lions compactness theorem for functions discretized in both space and time.
On the other hand, we consider only space-discretized functions in Lemma \ref{discrete_Aubin}.
Since Lemma \ref{discrete_Aubin} can be proved in the same manner as in \cite{G-L}, we omit its proof.
\end{remark}

\begin{proof}[Proof of Theorem \ref{AS_conv}]
For each $\delta > 0$, it follows from Lemmas \ref{AS_lem_esti1} and \ref{AS_lem_esti2} that
\begin{itemize}
\item
$\{ z_{\delta}^{(n)} \}$ and $\{ v_{\delta}^{(n)} \}$ are bounded in $L^\infty(0,T; H)$;
\item
$\{ \widetilde{z}_{\delta}^{(n)} \}$ and $\{ \widetilde{v}_{\delta}^{(n)} \}$ are bounded in $L^2(0,T; H)$;
\item
$\{ \partial_t z_{\delta}^{(n)} \}$ is bounded in $L^2(0, T; X^\ast)$.
\end{itemize}
Therefore, Lemma \ref{discrete_Aubin} implies that there exist a subsequence $\{ n_j \}$ of $\{ n \}$, $v_\delta \in L^\infty(0, T; H) \cap L^2(0, T; X)$, and $z_\delta \in L^\infty(0, T; H) \cap L^2(0, T; X) \cap W^{1,2}(0, T; X^\ast)$ such that 
\begin{alignat*}{2}
v_{\delta}^{(n_j)} &\to v_\delta, \quad z_{\delta}^{(n_j)} \to z_\delta
&&\qquad \text{weakly$\ast$ in $L^\infty(0, T; H)$},\\
\widetilde{v}_{\delta}^{(n_j)} &\to \partial_x v_\delta 
&&\qquad \text{weakly in $L^2(0, T, H)$},\\
z_{\delta}^{(n_j)} &\to z_\delta 
&&\qquad \text{strongly in $L^2(0, T; H)$},\\
\partial_t z_{\delta}^{(n_j)} &\to \partial_t z_\delta
&&\qquad \text{weakly in $L^2(0, T; X^\ast)$}
\end{alignat*}
as $j \to \infty$.
Since $h^{-1}$ is Lipschitz continuous, we have
\begin{alignat*}{2}
&v_{\delta}^{(n_j)} \to v_\delta
&&\qquad \text{strongly in $L^2(0,T; H)$ as $j \to \infty$},
\end{alignat*}
and hence $z_\delta = h(v_\delta)$. 

We show that for each $\delta > 0$, $v_\delta$ is a weak solution of (P)($\rho_\delta, v_0$).
Let $\eta \in W^{1,2}(0, T; H) \cap L^2(0, T; X)$ satisfy $\eta(T) = 0$.
Similarly to the proof of \eqref{AS_esti6}, for each $j$, we have
\begin{align*}
&\int^T_0 \int^1_0 \partial_t z_{\delta}^{(n_j)} \eta dx dt \\
&= - \int^T_0 \sum^{n_j-1}_{i=1} \int^{x_i^{(n_j)}}_{x_{i-1}^{(n_j)}} \bigg( \frac{v_{\delta, i+1}^{(n_j)} - v_{\delta, i}^{(n_j)}}{\Delta x^{(n_j)}} 
+ b\left(\frac{v_{\delta, i}^{(n_j)} + v_{\delta, i+1}^{(n_j)}}{2}\right) \rho^{(n_j)}_{\delta, i} \bigg) \widetilde{\eta}^{(n_j)} dx dt,
\end{align*}
where $\Delta x^{(n_j)} = 1/n_j$ and 
\begin{align*}
\widetilde{\eta}^{(n_j)}(t, x) \coloneqq \frac{\eta(t, x + \Delta x^{(n_j)}) - \eta(t, x)}{\Delta x^{(n_j)}}
\qquad \text{for $(t, x) \in Q(T)$}.
\end{align*}
It follows from integration by parts that
\begin{align*}
\int^T_0 \int^1_0 (\partial_t z_{\delta}^{(n_j)}) \eta dx dt 
&= -\int^T_0 \int^1_0 z_{\delta}^{(n_j)} \partial_t \eta dt dx
- \int^1_0 z_{\delta}^{(n_j)}(0) \eta(0) dx\\
&\to -\int^T_0 \int^1_0 z_\delta \partial_t \eta dx dt
- \int^1_0 z_\delta(0) \eta(0) dx 
\end{align*}
as $j \to \infty$.
We note that $z_\delta(0) = h(v_\delta(0)) = h(v_0)$.
In addition, we obtain
\begin{align*}
\int^T_0 \sum^{n_j-1}_{i=1} \int^{x_i^{(n_j)}}_{x_{i-1}^{(n_j)}} \frac{v_{\delta, i+1}^{(n_j)} - v_{\delta, i}^{(n_j)}}{\Delta x^{(n_j)}} \widetilde{\eta}^{(n_j)} dx dt
&= \int^T_0 \int^1_0 \widetilde{v}_\delta^{(n_j)} \widetilde{\eta}^{(n_j)} dxdt\\
&\to \int^T_0 \int^1_0 \partial_x v_\delta \partial_x \eta dx dt
\end{align*}
as $j \to \infty$.
For each $j$, we get
\begin{align*}
&\int^T_0 \sum^{n_j-1}_{i=1} \int^{x_i^{(n_j)}}_{x_{i-1}^{(n_j)}}  b\left(\frac{v_{\delta, i}^{(n_j)} + v_{\delta, i+1}^{(n_j)}}{2}\right) \rho^{(n_j)}_{\delta, i} \widetilde{\eta}^{(n_j)} dx dt \\
&= \int^T_0 \int_0^{1-\Delta x^{(n_j)}} b\left(\frac{v_{\delta}^{(n_j)} + v_{\delta}^{(n_j)} (\cdot, \cdot + \Delta x^{(n_j)})}{2} \right) \rho_{\delta}^{(n_j)} \widetilde{\eta}^{(n_j)} dx dt\\
&= \int^T_0 \int^{1-\Delta x^{(n_j)}}_0 b\left(\frac{v_{\delta}^{(n_j)} + v_{\delta}^{(n_j)}(\cdot, \cdot + \Delta x^{(n_j)})}{2} \right) \rho_{\delta}^{(n_j)} (\widetilde{\eta}^{(n_j)} - \partial_x \eta) dx dt\\
&\hspace{10mm} + \int^T_0 \int^{1-\Delta x^{(n_j)}}_0 \left( b\bigg(\frac{v_{\delta}^{(n_j)} + v_{\delta}^{(n_j)}(\cdot, \cdot + \Delta x^{(n_j)})}{2} \bigg) - b(v_\delta)\right) \rho_{\delta}^{(n_j)} \partial_x \eta dxdt\\
&\hspace{10mm} + \int^T_0 \int^{1-\Delta x^{(n_j)}}_0 b(v_\delta) (\rho_{\delta}^{(n_j)} - \rho_\delta) \partial_x \eta dxdt
-\int^T_0 \int^1_{1 - \Delta x^{(n_j)}}  b(v_\delta) \rho_\delta \partial_x \eta dx dt\\
&\hspace{10mm} +\int^T_0 \int^1_0 b(v_\delta) \rho_\delta \partial_x \eta dx dt\\
&\eqqcolon I_{1, j} + I_{2, j} + I_{3, j} + I_{4, j} + \int^T_0 \int^1_0 b(v_\delta) \rho_\delta \partial_x \eta dx dt.
\end{align*}
Since $\rho_{\delta}^{(n_j)}$ is uniformly bounded in $L^\infty(Q(T))$ with respect to $j$, $b$ is Lipschitz continuous, and $\rho_{\delta}^{(n_j)} \to \rho_\delta$ and $\widetilde{\eta}^{(n_j)} \to \partial_x \eta$ in $L^2(0,T; H)$ as $j \to \infty$, we have
\begin{align*}
I_{1, j} &\leq C_{b} \left\{ \int^T_0 |\rho_{\delta}^{(n_j)}|^2_H dt \right\}^\frac{1}{2} 
\left\{ \int^T_0 |\partial_x \eta - \widetilde{\eta}^{(n_j)} |_H dt \right\}^\frac{1}{2} \to 0,\\
I_{2, j} &\leq C_{b} |\rho_{\delta}^{(n_j)}|_{L^\infty(Q(T))} \int^T_0 \int^{1-\Delta x^{(n_j)}}_0 |\partial_x \eta| \left|\frac{v_{\delta}^{(n_j)}(\cdot, \cdot + \Delta x^{(n_j)}) + v_{\delta}^{(n_j)}}{2} - v_\delta \right| dxdt\\
&\leq \frac{C_{b} |\rho_{\delta}^{(n_j)}|_{L^\infty(Q(T))}}{2} \int^T_0 \int^{1-\Delta x^{(n_j)}}_0 |\partial_x \eta| (|v_{\delta}^{(n_j)} - v_{\delta}| + |v_{\delta}^{(n_j)}(\cdot, \cdot + \Delta x^{(n_j)}) - v_{\delta}|) dxdt\\
&\leq \frac{C_{b} |\rho_{\delta}^{(n_j)}|_{L^\infty(Q(T))}}{2} \int^T_0 \int^{1-\Delta x^{(n_j)}}_0 |\partial_x \eta| (|v_{\delta}^{(n_j)} - v_{\delta}| \\
&\hspace{10mm} + |v_{\delta}^{(n_j)}(\cdot, \cdot + \Delta x^{(n_j)}) - v_{\delta}(\cdot, \cdot + \Delta x^{(n_j)})|
+ |v_{\delta}(\cdot, \cdot + \Delta x^{(n_j)}) - v_{\delta}|) dxdt
\to 0, \\
I_{3, j}
&\leq C_{b} \int^T_0 \int^1_0 |\partial_x \eta| | \rho_{\delta}^{(n_j)} - \rho_\delta| dx dt
\to 0
\end{align*}
as $ j \to 0$.
Moreover, a simple computation gives
\begin{align*}
I_{4, j} \to 0 
\qquad \text{as $ j \to 0$}.
\end{align*}
Thus, we conclude that $v_\delta$ is a weak solution of (P)($\rho_\delta, v_0$).
For $\delta > 0$, since $\rho_\delta \in L^4(0,T; H)$, Theorem \ref{main_thm} implies that $v_\delta$ is a unique weak solution of (P)($\rho_\delta, v_0$).
Hence, the above convergences hold not only for certain subsequences but also for the whole sequences, that is, 
\begin{alignat*}{2}
&v_{\delta}^{(n)} \to v_\delta
&&\qquad \text{strongly in $L^2(0,T; H)$ and weakly$\ast$ in $L^\infty(0,T; H)$},\\
&\widetilde{v}_{\delta}^{(n)} \to \partial_x v_\delta
&&\qquad \text{weakly in $L^2(0,T; H)$},\\
&z_{\delta}^{(n)} \to z_\delta
&&\qquad\text{weakly in $W^{1,2}(0,T; X^\ast)$ and weakly$\ast$ in $L^\infty(0,T; X)$},\\
&\widetilde{z}_{\delta}^{(n)} \to \partial_x z_\delta
&&\qquad\text{weakly in $L^2(0,T; H)$}
\end{alignat*}
as $n \to \infty$.

Finally, we show the convergence of $\{ v_\delta \}$ to a weak solution of (P)($p, v_0$).
We remark that $C_{15}$ and $C_{17}$ given in Lemmas \ref{AS_lem_esti1} and \ref{AS_lem_esti2}, respectively, are independent of $\delta$ and $n$.
Thus, by letting $n \to \infty$ in \eqref{AS_esti1}, \eqref{AS_esti4}, and \eqref{AS_esti7}, we see that $\{ v_\delta \}$ is bounded in $L^\infty(0,T; H)$ and $L^2(0,T, X)$ and that $\{ z_\delta \}$ is bounded in $L^\infty(0,T; H), L^2(0,T, X)$ and $W^{1,2}(0,T; X^\ast)$.
Due to the standard Aubin--Lions compact theorem, there exist a subsequence $\{ \delta_j \}$ of $\{ \delta \}$, $v \in L^\infty(0, T; H) \cap L^2(0, T; X)$, and $z \in L^\infty(0, T; H) \cap L^2(0, T; X) \cap W^{1,2}(0, T; X^\ast)$ such that
\begin{alignat*}{2}
v_{\delta_j} &\to v
&&\qquad \text{weakly in $L^2(0, T; X)$ and weakly$\ast$ in $L^\infty(0, T; H)$}, \\
z_{\delta_j} &\to z
&&\qquad \text{weakly in $W^{1,2}(0, T; X^\ast)$ and strongly in $L^2(0, T; H)$}
\end{alignat*}
as $j \to \infty$.
From the Lipschitz continuity of $h^{-1}$, it follows that
\begin{alignat*}{2}
&v_{\delta_j} \to v
&&\qquad \text{strongly in $L^2(0,T; H)$ as $j \to \infty$},
\end{alignat*}
and hence $z = h(v)$.
By taking a suitable subsequence, we have $v_{\delta_j} \to v$ a.e. on $Q(T)$ as $j \to \infty$.
Let $\eta \in W^{1,2}(0,T; H) \cap L^2(0,T, X)$ satisfy $\eta(T) = 0$.
Then, for each $j$, we obtain
\begin{align*}
&- \int_{Q(T)} h(v_{\delta_j}) \partial_t \eta dxdt
+ \int_{Q(T)} (\partial_x v_{\delta_j} + b(v_{\delta_j}) \rho_{\delta_j}) \partial_x \eta dxdt
= \int^1_0 h(v_0) \eta(0) dx.
\end{align*}
The convergences of $\{ v_{\delta_j} \}$ and $\{ z_{\delta_j} \}$ give
\begin{align*}
\int_{Q(T)} h(v_{\delta_j}) \partial_t \eta dxdt 
&\to \int_{Q(T)} h(v) \partial_t \eta dxdt, \\
\int_{Q(T)} \partial_x v_{\delta_j} \partial_x \eta dx dt
&\to \int_{Q(T)} \partial_x v \partial_x \eta dx dt
\end{align*}
as $j \to \infty$.
Since $\rho_{\delta_j} \to p$ in $L^2(0,T; H)$ as $j \to \infty$, the Lebesgue dominated convergence theorem yields
\begin{align*}
&\left| \int_{Q(T)} (b(v_{\delta_j}) \rho_{\delta_j} \partial_x \eta - b(v) p \partial_x \eta) dx dt \right|\\
&\leq C_{b} \int^T_0 |\partial_x \eta|_H |\rho_{\delta_j} - p|_H dt
+ \int_{Q(T)}  |\partial_x \eta| |p| |b(v_{\delta_j}) - b(v)| dxdt\\
&\to 0
\end{align*}
as $j \to \infty$.
Therefore, we conclude that $v$ is a weak solution of (P)($p, v_0$).
In addition, if $p \in L^4(0, T; H)$, then the above convergence of $\{ v_{\delta_j} \}$ holds for the whole sequence $\{ v_\delta \}$ by virtue of Theorem \ref{main_thm}.
\end{proof}

\begin{center}
{\bf Acknowledgement }
\end{center}
 
The authors are grateful to members of Ebara Corporation for showing several interesting phenomena related to moisture penetration with fruitful discussion, which triggered this work.

\bibliographystyle{plain}

\end{document}